\documentclass[]{amsart}

\usepackage[utf8]{inputenc}
\usepackage[OT2,T1]{fontenc}
\DeclareSymbolFont{cyrletters}{OT2}{wncyr}{m}{n}
\DeclareMathSymbol{\Sha}{\mathalpha}{cyrletters}{"58}
\usepackage{graphicx}
\graphicspath{ {./images/} }

\usepackage[hyphens,spaces,obeyspaces]{url}
\usepackage[colorlinks,allcolors=blue,hyperindex,breaklinks]{hyperref}
\hypersetup{
           breaklinks=true,   
           colorlinks=true,   
           pdfusetitle=true,  
        }
   
\usepackage{orcidlink}
\title[Parameters of solvable automorphic forms]{Parameters of solvable automorphic forms}
\date{\today}
\usepackage[foot]{amsaddr}
\author{Peter Vang Uttenthal \orcidlink{0009-0001-0878-8213}}
\email{petervang@math.au.dk}
\address{Department of Mathematics, Aarhus University, Ny Munkegade 118, Building 1530-421, DK-8000
Aarhus C, Denmark}
\newcommand{\Gal}{\operatorname{Gal}}

\newcommand{\Q}{\mathbb{Q}}
\newcommand{\Z}{\mathbb{Z}}
\newcommand{\F}{\mathbb{F}}

\newcommand{\Ad}{\operatorname{Ad}^0(\overline{\rho})}

\usepackage{amsthm,amsmath, mathrsfs, mathtools}
\usepackage{amsfonts}
\usepackage{amssymb}
\usepackage{fancyhdr}
\usepackage{IEEEtrantools}
\usepackage{tikz-cd}
\usepackage[english]{babel}
\usepackage[utf8]{inputenc}

\usepackage{csquotes}

\newtheorem{theorem}{Theorem}
\newtheorem{lemma}[theorem]{Lemma}
\newtheorem{definition}[theorem]{Definition}
\newtheorem{proposition}[theorem]{Proposition}
\newtheorem{corollary}[theorem]{Corollary}
\newtheorem{remark}[theorem]{Remark}

\usepackage[hyphens,spaces,obeyspaces]{url}
\usepackage[colorlinks,allcolors=blue,hyperindex,breaklinks]{hyperref}

\hypersetup{
           breaklinks=true,   
           colorlinks=true,   
           pdfusetitle=true,  
        }
\usepackage{tabularx}
\usepackage{subcaption}
\begin{document}

\begin{abstract} 
In a letter from Tate to Serre dated March 26, 1974, Tate suggested a classification of weight one modular forms of prime level in terms of their associated odd Artin representations. This paper carries out an analogous classification of Maass wave forms of prime power level in terms of complex even representations. The parameters are identified with techniques from class field theory and Galois representations. The classification reveals that there exist distinct Maass cusp forms of tetrahedral type on $\Gamma_1(\ell)$ that remain inequivalent modulo $3$ for $\ell = 7687, 16363$ and $20887$, and that these $\ell$  are the three smallest such primes. 
\end{abstract}

\maketitle

\tableofcontents

\section{Introduction}
As a reference to F. Klein's classification of the finite subgroups of the rotation group of a sphere, Langlands' theory of base change for automorphic forms on $\operatorname{GL}(2)$ gave rise to the notion of an automorphic form of solvable type \cite{langlands}: An automorphic form on $\operatorname{GL}(2)$ over $\Q$ is said to be of solvable type (or just solvable) if it is attached to a complex Galois representation whose projective image is a solvable group.
For an integer $n\geqslant 0$, let $\Gamma_1(n)$ be the subgroup of $\operatorname{SL}(2,\Z)$ of matrices whose diagonal entries are congruent to $1$ modulo $n$ and whose lower left entry is congruent to $0$ modulo $n$. 
The starting point of the present paper is the question:
Do there exist distinct solvable Maass forms $\pi$ and  $\pi'$ on $\Gamma_1(\ell)$ for some prime $\ell$ such that $\pi$ and $\pi'$ remain distinct modulo $3$?
The notion of reduction modulo 3 will be clarified below.  
An automorphic form on $\Gamma_1(n)$ $(n\in \Z_{\geqslant 0})$ is said to be of \emph{level} $\Gamma_1(n)$. 
In this paper, the Maass forms considered will always be algebraic Maass forms of eigenvalue $\lambda = 1/4$.
If $\pi=\pi(\rho)$ is an algebraic Maass form corresponding to a complex, 2-dimensional represention of the absolute Galois group over $\Q$, then $\rho$ is ramified only at $\ell$ if and only if $\pi$ has level $\Gamma_1(\ell^k)$ for some $k\in \Z_{\geqslant 1}$, and 
in the affirmative, $\ell^k$ is the conductor of the Galois representation $\rho$. However, we will show that the projectivization $\tilde{\rho}$
must have conductor $\ell$ in the tetrahedral case
and conductor $\ell$ or $\ell^2$ in the octahedral case, respectively.

To answer the question, it is helpful to ask more generally: 
For which primes $\ell$ is there an integer $k\geqslant 1$ such that
there exists a solvable Maass form of level $\Gamma_1(\ell^k)$?  
In this paper, we answer both of these questions by giving an explicit classification of all solvable Maass forms on $\Gamma_1(\ell^k)$ for $k\in \{1,2\}$ in terms of their attached Galois representations. 
In addition, we give formulas for the number of distinct Maass forms of tetrahedral type on $\Gamma_1(\ell)$ and the number of Maass forms of octahedral type on $\Gamma_1(\ell^{k})$ for $k\in \{1,2\}$
in terms of dimensions of certain global Galois cohomology groups.

In a letter from Tate to Serre, dated March 26th, 1974, Tate suggested a classification of modular forms of weight one and level $\Gamma_0(\ell)$ where $\ell$ is a prime, by classifying
the corresponding complex, 2-dimensional \emph{odd} Galois representations of prime conductor $\ell$ \cite[p. 244]{Serre}.  
This paper carries out an analogous explicit classification of
Maass wave forms of level $\Gamma_1(\ell^k)$
of solvable type for $k \in \{1,2\}$ by 
classifying their associated complex \emph{even} Galois representations of prime power conductor.

Let $G_\Q =  \Gal(\overline{\Q}/\Q)$ be the absolute Galois group over $\Q$. 
If $\tilde{\rho}: G_\Q \to \operatorname{PGL}(2,\mathbb{C})$ is a projective representation,
a $lifting$ of $\tilde{\rho}$ is a representation $\rho: G_\Q \to \operatorname{GL}(2,\mathbb{C})$ such that the diagram
\[
\begin{tikzcd}
    \Gal(\overline{\Q}/\Q) \arrow[r, "\rho"] \arrow[dr, "\tilde{\rho}", swap] & \operatorname{GL}(2,\mathbb{C}) \arrow[d] \\
    &  \operatorname{PGL}(2,\mathbb{C})
\end{tikzcd}
\]
commutes. Throughout this paper, representations of $G_\Q$ will always implicitly be assumed to be continuous in the sense of having an open kernel.   
The \emph{conductor} of a projective representation $\tilde{\rho}$ of $G_\Q$ is the integer
$$
N = \prod_p p^{m(p)}
$$
where the product is over primes $p$ and where $m(p)$ is the least integer greater than or equal to zero such that the restriction $\tilde{\rho}|_{\Gal(\overline{\Q}_p/\Q_p)}$ has a lifting with conductor $p^{m(p)}$. By a theorem of Tate, if $\tilde{\rho}$ has conductor $N$, it has a lifting $\rho$ with conductor $N$ (and every lifting has conductor a multiple of $N$) \cite[Theorem 5]{Serre}. In this paper, 
whenever $\tilde{\rho}$ is a projective representation of $G_\Q$, $\rho$ will denote a lifting of $\tilde{\rho}$ with the same conductor as $\tilde{\rho}$.

Let $\tilde{\rho}: G_\Q \to \operatorname{PGL}(2,\mathbb{C})$ be a projective representation. For any place $p$ of $\Q$, let $D_p = \Gal(\overline{\Q}_p/\Q_p)$ and let $I_p$ be the inertia subgroup of $D_p$. 
Since $\tilde{\rho}$ is continuous, $\tilde{\rho}(G_\Q)$, is a finite subgroup of $\operatorname{PGL}(2,\mathbb{C})$. 
If $\tilde{\rho}$ is unramified at $p$, then $m(p)=0.$
If $\tilde{\rho}$ is tamely ramified at $p$, then $m(p)=1$ or $m(p)=2$. 
In the sequel, the following lemma by Serre will be needed:
Let $e_p$ be the cardinality of $\tilde{\rho}(I_p)$. If $\tilde{\rho}$ is tamely ramified at $p$ and 
if $e_p \geqslant 3$, then the conductor of $\tilde{\rho}$ is exactly divisible by $p$ if and only if
$e_p$ divides $p-1$ \cite[p. 248]{Serre}.
Note that if $K$ is the number field fixed by the kernel of $\tilde{\rho}$, then 
$\tilde{\rho}(I_p)$ is the inertia group and $\tilde{\rho}(D_p)$ is the decomposition group at $p$
in $\Gal(K/\Q)$, respectively. In particular 
$e_p$ is the ramification index of $p$ in $K/\Q$. 

In Langlands' theory of base change, complex 2-dimensional Galois representations $\rho$ of solvable type are shown to correspond to certain automorphic forms on $\operatorname{GL}(2)$ over $\Q$; we will now define the notion of being of solvable type precisely. 
 By a theorem of Klein, any finite subgroup of $\operatorname{PGL}(2,\mathbb{C})$ is isomorphic to one of the following groups \cite[Proposition 4.9.1]{Bump}:
\begin{enumerate}
    \item The cyclic group of order $n$.
    \item The dihedral group $D_{2n}$ for some $n \in \Z_{\geqslant 2}$.
    \item The alternating group $A_4$ of rotations of the tetrahedron. 
    \item The symmetric group $S_4$ of rotations of the octahedron. 
    \item The alternating group $A_5$ of rotations of the icosahedron.
\end{enumerate}
Consequently, complex 2-dimensional Galois representations are classified according to the isomorphism type of their projective image in $\operatorname{PGL}(2,\mathbb{C})$, see e.g. \cite{langlands}. 
The Langlands-Tunnell theorem applies to \emph{solvable} groups, and all groups in the list except $A_5$ are solvable. The cyclic and dihedral cases have been treated in classical work going back at least to Hecke. In this paper, we focus on representations whose projective image is isomorphic to $A_4$ or $S_4$; such representations present a sweet spot where algebraic techniques can be applied to prove new facts about Maass wave forms, and this is the purpose of this paper. 

A complex 2-dimensional  Galois representation 
$\rho$ is said to be attached to a Maass form $\pi$ if the automorphic $L$-function $L(\pi,s)$ of the Maass form is equal  to the Artin $L$-function $L(\rho,s)$ of $\rho$. In this  case,
for all finite places $v$ at which  $\rho$ is unramified, 
the trace of $\rho(\sigma_v)$ (where $\sigma_v$ is a Frobenius automorphism   at $v$) is equal  to the  $v$th Fourier coefficient of $\pi$. A Maass form $\pi$ is of tetrahedral or octahedral type if it is attached to a complex 2-dimensional Galois representation $\rho$ whose projective image $\tilde{\rho}(G_\Q)$ is isomorphic to $A_4$ or $S_4$, respectively.

Let $\Q$ be the field of rational numbers. For any positive integer $n$, let $\zeta_n$ be a primitive $n$th root of unity, and let $\Q(\zeta_n)$ be the cyclotomic extension of $\Q$ obtained by adjoining $\zeta_n$.

\begin{theorem}[Classificaton of Maass forms of tetrahedral type on $\Gamma_1(\ell)$] \label{1}
Let $I$ be the set of primes $\ell$ such that 
\begin{enumerate}
    \item $\ell \equiv 1 \bmod 3$ 
    \item The class number $h_L$ of the cubic subfield $L$ of $\Q(\zeta_\ell)$ is even.
\end{enumerate}
For each $\ell \in I$, there is a Maass wave form $\pi^{(\ell)}$ on $\Gamma_1(\ell)$
attached to a complex 2-dimensional representation of $G_\Q$ of tetrahedral type and of conductor $\ell$. 
The list $\{ \pi^{(\ell)}: \ell \in I\}$ is exhaustive: Any Maass form of tetrahedral type on $\Gamma_1(\ell)$ for a prime $\ell$ occurs in this list. 
\end{theorem}
For a number field $E$ and a set of places $S$ in $\Q$, let $E_S$ be the maximal extension of $E$ unramified outside the primes in $E$ above $S$. In particular, $E_\emptyset$ is the maximal unramified extension of $E$.
\begin{corollary}
For any $\ell \in I$, let $n_\ell^{(A_4)}$ denote the number of tetrahedral Maass forms on $\Gamma_1(\ell)$. 
If we define 
\[ k_\ell = \frac{1}{2}\dim_{\F_2} H^1(\Gal(L_\emptyset/L), \F_2) \]
then 
\[
n_\ell^{(A_4)} = 2^{k_\ell}-1.
\]
\end{corollary}

In the field diagram below, we show for a prime $\ell \equiv 1 \bmod 3$
the totally real cubic subfield $L$ of $\Q(\zeta_\ell)$. In this diagram, $\overline{\zeta_\ell}$ denotes the conjugate of $\zeta_\ell$, and the field $\Q \left( \zeta_{\ell} + \overline{\zeta_{\ell}} \right)$ is the subfield of $\Q(\zeta_\ell)$ fixed by complex conjugation. 
\[
\begin{tikzcd}[arrows=dash]
\Q(\zeta_{\ell}) \\
\Q \left( \zeta_{\ell} + \overline{\zeta_{\ell}} \right) \arrow[u] \\ 
L \arrow[u]\\
\Q \arrow[u, swap, "(L:\Q)=3"]
\end{tikzcd}
\]

We prove an analogous theorem for Maass forms of octahedral type.
For any set of places $S$ in $\Q$, a number field $E$, and a prime $p$, we let $E_S^{(p)}$ be the maximal $p$-elementary abelian extension of $E$ unramified outside the places in $E$ above $S$. 

\begin{theorem} \label{2}
Let $J$ be the set of primes $\ell$ such that 
\begin{enumerate}
    \item $\ell \equiv 1 \bmod 4$
    \item The class number of $\Q(\sqrt{\ell})$ is divisible by 3. In particular, there is an unramified cubic extension $L/\Q(\sqrt{\ell})$ such that 
    $\Gal(L/\Q) \simeq S_3.$
    \item \label{cond3}  When regarding $\Gal(L_{\{ \ell \} }^{(2)}/L)$ as a module over $\F_2[\Gal(L/\Q(\sqrt{\ell}))] \simeq \F_2[\Z/3\Z],$ it contains an irreducible 2-dimensional 
    representation.  
\end{enumerate}
For each $\ell \in J$, there is a Maass wave form $\pi^{(\ell)}$ on $\Gamma_1(\ell^k)$ for some $k\in \{ 1,2\}$ attached to a complex 2-dimensional representation $\rho$ of $G_\Q$ with $\tilde{\rho}(G_\Q) \simeq S_4$ and of conductor $\ell^k$. 
The list $\{ \pi^{(\ell)}: \ell \in J\}$ is exhaustive: Any Maass form of octahedral type on $\Gamma_1(\ell^k)$ for some $k \in \{1,2\}$ occurs in this list. 
\end{theorem}

Let $\ell \equiv 1 \bmod 4$ with 3 dividing the class number of $\Q(\sqrt{\ell})$, and let $L/\Q(\sqrt{\ell})$ be an unramified cubic extension as in Theorem \ref{2}.
The relevant number fields of Theorem \ref{2}, as well as some of their Galois groups, are depicted in the diagram below. The field $L^{(2)}_{\{\ell \}}$ is the maximal 2-elementary abelian extension of $L$ unramified outside the primes in $L$ above $\ell$. 
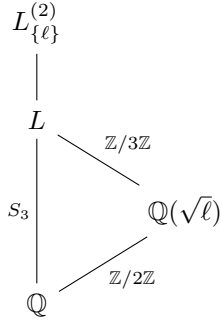
\[
\begin{tikzcd}[arrows=dash]
L^{(2)}_{\{ \ell \} } \\
L \arrow[u] \\ 
& \Q(\sqrt{\ell}) \arrow[ul, "\Z/3\Z", swap]\\
\Q \arrow[uu, "S_3"] \arrow[ur, "\Z/2\Z", swap]
\end{tikzcd}
\]

The condition (\ref{cond3}) in Theorem \ref{2} can be equivalently stated: 
In the Jordan-H\"{o}lder series of the Galois module $\Gal(L^{(2)}_{\{\ell\}}/L)$ over the group ring $\F_2[\Z/3\Z]$, at least one irreducible representation of $\Z/3\Z$ 
valued in $\operatorname{GL}(2,\F_2)$ occurs as a composition factor.

\begin{corollary}
Fix $\ell \in J$ and let $n_\ell^{(S_4)}$ denote the number of octahedral Maass forms on $\Gamma_1(\ell^k)$ for $k\in \{ 1,2\}$.  
Let $\mathscr{L}$ be the family of unramified cubic extensions of $F := \Q(\sqrt{\ell})$. 
For 
$
h_\ell := \dim_{\F_3} H^1(\Gal(F_{\emptyset}/F) ,\F_3), 
$
$
\operatorname{card} \mathscr{L} = (3^{h_\ell} - 1)/2.
$
For each $L \in \mathscr{L}$, let
$$
k_L =  \frac{\dim_{\F_2} H^1(\Gal ( L_{\{ \ell\} }/L) ,\F_2) -  \dim_{\F_2} H^1(\Gal(F_{\{ \ell\} }/F),\F_2)}{2}.
$$
Then 
\begin{IEEEeqnarray*}{rCr}
n^{(S_4)}_\ell &=& 
 \sum_{L \in \mathscr{L}}\left(2^{k_L}-1  \right).
\end{IEEEeqnarray*}
\end{corollary}

Let $I$ be the set of primes $\ell \equiv 1 \bmod 3$ such that 2 divides the class number of the totally real cubic subfield of $\Q(\ell)$, and for $\ell \in I$, let $\pi^{(\ell)}$ be a Maass form of tetrahedral type on $\Gamma_1(\ell)$ as in Theorem \ref{1}. By numerical experimentation in magma and pari-gp, Theorem \ref{1} reveals the smallest prime $\ell=7687$ for which there exist at least two inequivalent tetrahedral Maass forms on $\Gamma_1(\ell)$ that remain inequivalent modulo 3, thus answering the question asked above. Before we state this result in detail, we comment on the notion of reducing the Maass form modulo 3. 
Since $A_4 \simeq \operatorname{PSL}(\F_3)$, a tetrahedral number field is the same as a projective 2-dimensional Galois representation $\tilde{\rho}$ over the finite field $\F_3$.
By Proposition \ref{IJNT}, $\tilde{\rho}$ can be lifted to a linear representation valued in $\operatorname{SL}(2,\F_3)$ of the same conductor. In particular, the Fourier coefficients at unramified primes in the associated 
Artin $L$-function can be defined as traces of matrices in $\operatorname{SL}(2,\F_3)$ 
that are images of Frobenius automorphisms. 
\begin{theorem} \label{distinct}
For each $\ell \in I$ with $\ell < 7687,$ the Maass form $\pi^{(\ell)}$ of tetrahedral type is unique modulo 3. 
Among $\ell \in I$ such that $\ell \leqslant 10^5$,
there exist exactly three tetrahedral Maass forms $\pi^{(\ell, i)}$ $(1\leqslant i \leqslant 3)$
of level $\Gamma_1(\ell)$ that are distinct modulo 3 for the eight primes 
$$\ell = 7687,  16363, 20887, 37087, 55609, 62617, 70597, 99529.$$
For all other $\ell \in I$ less than $10^5$ (of which there are $636$),
there exists exactly one tetrahedral Maass form $\pi^{(\ell)}$ modulo 3. 
\end{theorem}

In \cite{vangIJNT}, the following statement is applied to identify the smallest even $p$-adic Galois representations measured by conductor.  

\begin{theorem} If $\pi$ is any tetrahedral Maass form on $\Gamma_1(n)$ for some integer $n \geqslant 1$, then the smallest such $n \in \Z_{\geqslant 1}$ occurs when $n$ is prime, namely for $n=163$, and the second smallest such $n \in \Z_{\geqslant 1}$ is also prime, namely $n=277$. 
\end{theorem}

In the language of representation theory, Theorem \ref{distinct} has the following consequence:
There exist automorphic representations 
$\pi$ and $\pi'$ of $\operatorname{GL}(2)$ over the ring of adeles $\mathbb{A}_\Q$ over $\Q$
with decompositions into restricted tensor products
$$
\pi = \bigotimes_v \pi_v, \quad \pi'= \bigotimes_v \pi'_v,
$$
where $v$ runs over all places of $\Q$, and where $\pi_v$ and $\pi'_v$ are representations of the $v$-adic group $\operatorname{GL}(2,\Q_v)$ such that $\pi_v$ and $\pi'_v$ are spherical for all places $v$ except the place $v=\ell$ for a prime number $\ell,$ and the smallest such prime is $\ell = 7687.$

The Langlands-Tunnell theorem is (to the knowledge of the author) the only known case of reciprocity in the Langlands program relating automorphic forms with \emph{even} Galois representations \cite{Calegari} (a complex 2-dimensional Galois representation $\rho$ is said to be even if $\det \rho(c) =1$ for any complex congujation $c\in \Gal(\overline{\Q}/\Q)$ and is said to be odd otherwise). Automorphic forms on $\operatorname{GL}(2)$ over $\Q$ corresponding to complex, even Galois representations are the Maass wave forms \cite{gelbart}. Now, even representations are not known to come from geometry, and therefore much less is known about them and their associated automorphic forms compared to the odd Galois representations and their associated holomorphic forms. For this reason, we choose to focus on Maass forms in this paper, and we explore how much information we can prove about them when we know they are arithmetic.

Since we wish to study Maass forms, we begin by constructing a family of \emph{totally real} number fields $K/\Q$. We are interested in identifying a class of Maass forms of solvable type that we can list in an exhaustive way, and to make this classification problem tractable, this paper restricts attention to number fields ramified at exactly one prime $\ell$. 
Later, in Theorem \ref{diagonal}, we will prove directly that the smallest known Maass form of tetrahedral type is parameterized by a number field that happens to be 
ramified at just a single prime. 

We build the family of totally real fields $K/\Q$ from the ground field $\Q$ and up. 
The construction differs slightly according to whether we wish for $\Gal(K/\Q)$ to be isomorphic to $A_4$ or $S_4$. Although any $S_4$-extension $K/\Q$ contains an $A_4$-extension over a quadratic field $\Q(\sqrt{\ell}),$ this $A_4$-extension $K/\Q(\sqrt{\ell})$ does not descend to $\Q$. In this sense, the tetrahedral case is not a special case of the octahedral case. 

To construct tetrahedral extensions, we start by specifying a cubic, totally real extension $L/\Q$ ramified at a prime $\ell \equiv 1 \mod 3.$ In \cite{even1}, a cubic field $L/\Q$ 
ramified only at the single prime $\ell=349$ was taken from a list of number fields made by D. Shanks \cite{shanks}. A prime $\ell $ is said to be a Shanks prime if 
$\ell  = a^2 + 3a + 9$ for some integer $a \geqslant -1.$
Shanks primes have the nice property that they admit an explicit formula for the defining polynomial of the cubic totally real subfield of $\Q(\zeta_\ell).$
In  this paper, we observe that Shank's cubic fields are a special case of the family of cubic totally real subfields of the cyclotomic fields $\Q(\zeta_\ell)$ whenever $\ell \equiv 1 \bmod 3.$ This observation provides a wider range of cubic fields from which to build totally real $A_4$-extensions compared to the cubic fields tabulated in \cite{shanks}. 

\begin{remark}
    We conjecture that there are infinitely many primes $\ell$ which occur as the only ramified prime in a totally real extension $K/\Q$ of tetrahedral or octahedral type. 
    Furthermore, we believe the set of such primes have positive density. These conjectures, which are still open, are inspired by the Cohen-Lenstra heuristics. 
    Notice that except the congruence conditions, the conditions defining the primes $\ell$ in Theorem \ref{1} and \ref{2} are most likely not covered by the classical theorem of Chebotarev. Instead, we suggest to reformulate these conditions in terms of spin symbols of prime ideals (a notion first introduced in \cite{FIMR}).   
\end{remark}

\subsection{Related work} While the vision of 
R.P. Langlands was to classify all number fields in terms of automorphic representations of $\operatorname{GL}(n),$
the present work carries out the converse parametrization in a special case: We parameterize all automorphic representations of $\operatorname{GL}(2)$ over the field $\Q$ of rational numbers corresponding to solvable Maass forms in terms of number fields. 

In formulating his vision, Langlands was inspired by the work of E. Artin and others, who successfully completed the classification program for abelian extensions (i.e. Galois extensions with abelian Galois group). Indeed, it is a consequence of global class field theory that abelian extensions of a number field $F$ are parameterized by automorphic representations of $\operatorname{GL}(1)$ over $F$.

In his theory of base change \cite{langlands}, Langlands made progress on  
nonabelian extensions corresponding to solvable subgroups of $\operatorname{PGL}(2,\mathbb{C}),$  where the classification of number fields splits into distinct cases: Totally complex number fields are parameterized by modular forms, while totally real number fields are parameterized by Maass forms. 

By definition, Maass forms are purely analytic objects: they are functions of a complex variable that are not holomorphic but satisfy a partial differential equation, forcing them to be real analytic. Nevertheless, these analytic objects are parameterized by purely arithmetic data that are explicitly enumerated in this work. 
The parameters are number fields that are totally real, solvable, nonabelian, ramified at a single prime $\ell$, and defined by a polynomial of degree 12 or 24 over the field $\Q$ of rational numbers that can be computed explicitly. The results on Maass forms presented here are obtained by combining the Langlands-Tunnell theorem \cite{langlands} \cite{tunnell} with methods from explicit global class field theory.

\subsection{Acknowledgements}
I am grateful to Ravi Ramakrishna for sharing his ideas on the subject of this paper and to 
Frank Calegari for sharing detailed comments that helped improve the quality of the paper. 
I would like to thank Paul Nelson and Peter Sarnak for helpful conversations, and the reviewer for 
their careful readings and helpful suggestions that led to a greater precision.

This work was supported by a research grant (VIL54509) from Villum Fonden.

\section{Maass wave forms of tetrahedral type }
\subsection{Explicit construction by means of global class field theory}

\begin{definition}
If $L$ is a number field, $S$ is a finite set of places in $L$, and $p$ is a prime number, 
let $L_S^{(p)}$ be the $p$-Frattini extension of $L$
unramified outside $S$. This is the maximal $p$-elementary abelian extension of $L$ unramfied outside $S$.
\end{definition}

\begin{definition}
    A prime $\ell$ is a Shanks prime 
    with parameter $a \in \mathbb{Z}_{\geq -1}$
    if 
    $$
\ell = a^2 +3a +9.
    $$
\end{definition}

\begin{lemma} If $\ell$ is a Shanks prime with parameter $a$, the polynomial
$$
f(x) = x^3 - ax^2 -(a+3)x -1
$$
defines a cubic Galois extension $L/\mathbb{Q}$ of discriminant $
d_{L} = \ell^2.$ 
\end{lemma}
\begin{proof}
    See \cite{shanks}.
\end{proof}

In \cite{shanks}, Shanks proved that if $p \equiv 2 \mod 3$, 
and if $\ell$ is a Shanks prime corresponding to the cubic field $L$, then the $p$-rank of the class group of $L$ is even. Below, we generalize this theorem to a wider class of primes. 

\begin{lemma}
    Let $\ell$ be a prime such that $\ell \equiv 1 \mod 3$. Then $\Q(\zeta_\ell)$ contains a unique cubic number field $L$ such that  
    $L/\Q$ is totally real, totally ramified at $\ell$, and unramified at all other places.
\end{lemma}
\begin{proof} Since $\ell  \equiv  1 \bmod 3$, the cyclic group $\Gal(\Q(\zeta_\ell)/\Q) \simeq (\Z/\ell \Z)^\times \simeq \Z/(\ell-1)\Z$ contains a unique subgroup $H$ of order $(\ell-1)/3$. Let $L$ be the subfield fixed by $H$. Note that $L/\Q$ is totally real since any complex conjugation  $c \in  \Gal(\overline{\Q}/\Q)$ has order 2, and therefore the canonical image of $c$ in $\Gal(L/\Q) \simeq \Z/3\Z$ is trivial. 
\end{proof}

\begin{lemma}[Existence of totally real $A_4$-extentions ramified at one prime] \label{A4}
Let $\ell \equiv 1 \mod 3$ and
$L$ be the cubic subfield of $\Q(\zeta_\ell)$. 
Suppose the class number $h_L$ of $L$ is divisible by 4. 
Then there exists an unramified extension $K/L$ such that 
$K/\Q$ is Galois 
with $\Gal(K/\mathbb{Q}) \simeq A_4.$ 
\end{lemma}

\begin{proof}
First, we will show that the 2-part of the class group is not cyclic. 
Let $D^{(2)}$ denote the 2-Hilbert class field of $L$. 
The field $D^{(2)}$ is Galois over $\Q$ and, by assumption, 4 divides the order of $\Gal(D^{(2)}/L)$.  Suppose 
$$
\Gal(D^{(2)} / L) \simeq \mathbb{Z}/2^k\mathbb{Z}
$$
for some integer $k\geq 2$. 
Consider the exact sequence 
$$
1 \to \Gal(D^{(2)}/L) \to \Gal(D^{(2)}/\mathbb{Q}) \to \Gal(L/\mathbb{Q}) \to 1.
$$
Since $\Gal(L/\mathbb{Q})$ is a 3-group and 3 does not divide the order of
$\Gal(D^{(2)}/L)$, 
this sequence is split.
The action of $\Gal(L/\mathbb{Q})$ on $\Gal(D^{(2)}/L)$
can be viewed as a representation
$$
\mathbb{Z}/3 \to \mathrm{Aut}(\mathbb{Z}/2^k\Z) \simeq (\mathbb{Z}/2\Z) \times (\mathbb{Z}/2^{k-2}\Z),
$$
which must be trivial. It follows that $D^{(2)}/L$ descends to an extension 
 $D_0^{(2)}$ over $\Q$. Furthermore, $D^{(2)}_0/\Q$ does not pick  up ramification in the descent: Indeed, since $D^{(2)}/L$ is unramified, the ramification index $e(\ell,D^{(2)}/\Q)$ of $\ell$ in
$D^{(2)}/\Q$ cannot be divisible by any other primes than $3$; hence the same is true for $e(\ell, D^{(2)}_0/\Q)$. At the same time, 
$D^{(2)}_0/\Q$ is a 2-extension.  Hence, $D^{(2)}_0/L$ is unramified, but there are no unramified extensions of $\mathbb{Q}$, so we have reached a contradiction.
Consequently, the unramified 2-Frattini extension $L_\emptyset^{(2)}$ of $L$ 
has
$$
\dim_{\F_2} \Gal(L_\emptyset^{(2)}/L) \geq 2,
$$ 
and there is an unramified extension $K/L$
such that 
$$
\Gal(K/L) \simeq (\mathbb{Z}/2\Z)\times (\mathbb{Z}/2\Z).
$$
The field $K/\Q$ is Galois and $\Gal(K/\mathbb{Q})$ has order 12. 
There is a split exact sequence
$$
1 \to \Gal(K/L) \to \Gal(K/\mathbb{Q}) \to \Gal(L/\mathbb{Q}) \to 1.
$$
If $K/\mathbb{Q}$ was abelian, the conjugation action would be trivial. Then $K$ would descend to an unramified extension of $\mathbb{Q}$ of degree 4, a contradiction. Hence, $K/\mathbb{Q}$ is nonabelian. Finally, $\Gal(K/\mathbb{Q})$ has a normal Sylow-2 subgroup. By Table 1 in K. Conrad's note on groups of order 12 (\cite{conrad}), the only nonabelian group of order 12 with a unique Sylow-2 subgroup is $A_4$. We conclude that 
$
\Gal(K/\mathbb{Q}) \simeq A_4.
$
\end{proof}

In \cite{shanks}, Shanks proved that if $p \equiv 2 \mod 3$, 
and if $\ell$ is a Shanks prime corresponding to the cubic field $L$, 
then the $p$-rank of the class group of $L$ is even. 
Below, we generalize this theorem to a wider class of primes. 

\begin{theorem} \label{rank}
Assume that $\ell \equiv 1 \mod 3$.
Let $L$ be the totally real cubic subfield of $\Q(\zeta_\ell).$
If $p \equiv 2 \mod 3$, the $p$-rank of the class group of $L$ is even. In particular, the class number $h_L$ of $L$ is divisible by $p$ if and only if $h_L$ is divisible by $2p$. 
\end{theorem}

\begin{proof} There is an integer $n \geq 1$ such that 
$
\Gal(L^{(p)}_\emptyset/L) \simeq (\mathbb{Z}/p\Z)^n,
$
and we want to show that $n$ is even. The exact sequence
$$
1 \to \mathrm{Gal} (L^{(p)}_\emptyset/L) \to \Gal(L^{(p)}_\emptyset/\mathbb{Q}) \to \Gal(L/\mathbb{Q}) \to 1
$$
(where $\Gal(L/\mathbb{Q}) \simeq \mathbb{Z}/3\Z $ acts on the $\F_p$-vector space $\Gal(L^{(p)}_\emptyset/L)$) is split because $p$ is coprime to 3. 
For any prime $p$, the finite field $\mathbb{F}_p$ contains a primitive third root of unity if and only if 3 divides $p-1$. If $\mathbb{F}_p$ does not contain a primitive third root of unity, the only irreducible representations of $\Z/3\Z$ are the trivial 1-dimensional representation, and a nontrivial 2-dimensional representation $U_2$ (cf. \cite{repZ3}). Let $p \equiv 2 \mod 3$, and consider the decomposition of $\Gal(L^{(p)}_\emptyset/L)$ as an
$\mathbb{F}_p[\mathbb{Z}/3\Z]$-module,
$$
\Gal(L^{(p)}_\emptyset/L) \simeq U_2^k \oplus 
\F_p^m
$$
for some integers $k,m\geq0.$ Suppose 
$m \geq 1$. Let $D$ be the fixed field of $U_2^k$. Since $U_2^k$ is stable under conjugation by $\Gal(L/\mathbb{Q})$, it is normal in $\Gal(L^{(p)}_\emptyset/\mathbb{Q})$ and hence $D/\Q$ is Galois.
There is a split exact sequence
$$
1 \to \Gal(D/L) \to \Gal(D/\mathbb{Q}) \to \Gal(L/\mathbb{Q}) \to 1,
$$
and since $\Gal(D/L)$ is the sum of all the trivial representations, the semidirect product in the middle is direct. Hence, $D$ descends to $\mathbb{Q}$. Let  $D_0$ be the extension of $\Q$ such that the base change  
of $D_0$ to $L$ equals  $D$. 
Since $D/L$ is a $p$-extension while the degree of $L$ over $\Q$ equals $3 \neq p$, $D$ cannot  pick up ramification in the descent:  the ramification index 
$e(\ell, D_0/\Q) =1$. Hence, $D_0/\Q$ must be everywhere unramified, a contradiction. We conclude that $m=0$ and that 
$n=2k.$
\end{proof}

If $p\equiv 1 \mod 3$, the representation theory of the cyclic group of order 3 does not impose any constraints on the parity of the $p$-rank of the class group.
\begin{proposition} 
If $p \equiv 1 \mod 3$, 
then every  irreducible  submodule of the  
$\mathbb{F}_p [\mathbb{Z}/3\Z]$-module
$
\Gal(L_\emptyset^{(p)}/L) 
$
is $1$-dimensional. 
\end{proposition}
\begin{proof}
The field $\F_p$ contains a third root of unity, and, up to equivalence, there are three irreducible representations of 
$\Z/3\Z$ over $\F_p$, all 1-dimensional. 
\end{proof}

\begin{remark}
    As an example, consider $p=7 \equiv 1 \mod 3$. 
If $a=16$ and $\ell=313$, then the order of the class group is $h_L=7$, so the 7-rank is 1, while if $a=266$ and $\ell=71563$, and if $\operatorname{Cl}_L$ is the class group of $L,$ then
$$
\operatorname{Cl}_L \simeq (\Z/7\Z)\times (\Z/49\Z),
$$
so the 7-rank of the class group is 2.
\end{remark}

\begin{lemma}\label{counting}
    Let $G$ be a group, $V$ a vector space over a finite field $F$, and let $k\geq 1$ be an integer.
    Suppose $V$ is an irreducible $F[G]$-module. 
    The number of nontrivial irreducible $F[G]$-submodules of the direct sum $V^{\oplus k}$ is equal to the cardinality of the projective space $\mathbb{P}(F^k)$.
\end{lemma}
\begin{proof}
Any nontrivial irreducible $F[G]$-submodule of $V^{\oplus k}$
must be isomorphic to $V$ as $F[G]$-modules. 
By Schur's Lemma, any $\varphi \in \operatorname{Hom}_G(V, V^{\oplus k })$ 
has the form 
$$
\varphi(v)= (\lambda_1 v, \lambda_2v, \ldots, \lambda_k v), \quad v\in V
$$
for some $(\lambda_1,\ldots,\lambda_k) \in F^k$.
If $\varphi \in \operatorname{Hom}_G(V,V^{\oplus k})$ is defined by 
$(\lambda_1,\ldots,\lambda_k)\in F^k$, then the $F[G]$-submodule $M:= \varphi(V)$ is trivial if and only if  $\lambda_i = 0$ 
for all $i$ ($1\leq i \leq k$).

Let $\varphi, \psi \in \operatorname{Hom}_G(V,V^{\oplus k})$, and suppose their images are equal to the same space $M$:
$$
M = \varphi(V) = \psi(V).
$$
Then $ \psi|_{M}^{-1} \circ \varphi \in \operatorname{Hom}_G(V,V)$ is a scalar multiple of the identity by Schur's Lemma.  
Thus, it suffices to count $k$-tuples $(\lambda_1,\ldots, \lambda_k) \in F^{k}\setminus \{(0,\ldots,0)\}$,
where $(\lambda_1,\ldots,\lambda_k)$ and $(\lambda_1',\ldots,\lambda'_k)$ define the same $F[G]$-submodule if and only if they are proportional. The desired count is thus the cardinality of $\mathbb{P}(F^k)$. 
\end{proof}
If $A$ is a finite set, the cardinality of $A$ will be denoted as $\operatorname{card} A$.

\begin{proposition}\label{kA4}
Let $\ell \equiv 1 \mod 3$, and let $L$ be the cubic subfield of $\Q(\zeta_\ell)$. 
Suppose the 2-rank of the class group of $L$ is equal to $2k$ for an integer $k\geq 1$. Then there exist exactly 
$$
n_\ell^{(A_4)} :=\operatorname{card} \mathbb{P} (\F_2^k) = 2^k-1
$$
distinct $A_4$-extensions $K^{(i)}/\Q$  such that $K^{(i)}$ contains $L$ and $K^{(i)}/L$ is unramified $(1\leq i \leq 2^k-1)$.
\end{proposition}
\begin{proof}
By Theorem \ref{rank}, the 2-rank of the class group is even. As in the proof of Theorem \ref{rank},
$$
\Gal(L^{(2)}_\emptyset/L) \simeq U_2^{\oplus k} 
$$
as $\mathbb{F}_2[\mathbb{Z}/3\Z]$-modules, where $U_2$ is the 2-dimensional irreducible representation. 
By Lemma \ref{counting}, there are exactly $2^k -1$ nontrivial irreducible $\F_2[\Z/3\Z]$-submodules of $\Gal(L^{(2)}_\emptyset/L)$ (all of which are isomorphic to $U_2$).
Let $M$ be such an irreducible nontrivial $\F_2[\Z/3\Z]$-submodule. 
Since the order of $\Gal(L/\Q)$ is prime to the characteristic of $\F_2$, the $\F_2[\Z/3\Z]$-module $U_2^{\oplus k}$ is semisimple; hence $\Gal(L^{(2)}_\emptyset/L) = M \oplus M^\perp$ for some
$\F_2[\Z/3\Z]$-submodule $M^\perp$. 
Let $D$ be the fixed field of $M^\perp$. Since $M^\perp$ is  $\Gal(L/\mathbb{Q})$-stable, the field $D$ is Galois over $\mathbb{Q}$. 
As in the proof of Lemma \ref{A4}, $\Gal(D/\mathbb{Q})$ is a group of order 12 that has to be nonabelian and contains a normal Sylow-2 subgroup, and therefore has to be $A_4$. 

Finally, suppose $E$ is a Galois extension of $\mathbb{Q}$ with Galois group $A_4$ containing $L$ such that $E/L$ is unramified. 
Then $\Gal(E/L)$ has order 4 and must be the unique Sylow-2 subgroup of $A_4$, which is $(\mathbb{Z}/2\Z)^2$. It follows that $E$ is contained in the unramified $2$-Frattini extension $L_\emptyset^{(2)}$ of $L$. 
Thus $\Gal(E/L)$ must be one of the $2^k -1 $ nontrivial irreducible $\F_2[\Z/3\Z]$-submodules of $\Gal(L^{(2)}_\emptyset/L)$.
\end{proof}

\subsection{Constructing linear mod $3$ representations}
Let $I$ be the set of primes $\ell \equiv 1 \bmod 3$ such that 
the class number of the totally real cubic subfield of $\Q(\zeta_\ell)$ is even. By Proposition \ref{kA4}, there is a number field $K/\Q$ ramified at $\ell$ and unramified at all other places such that $\Gal(K/\Q) \simeq \operatorname{PSL}(2,\F_3)$. 
The natural projection $\Gal(\overline{\Q}/\Q) \to \Gal(K/\Q)$ can be regarded as a projective representation 
\[
\tilde{
\rho}^{(\ell)}: \Gal(\overline{\Q}/\Q) \to \operatorname{PSL}(2,\F_3)
\]
Since the ramification index $e(\ell,K/\Q)$ in $K/\Q$ is 3, which divides $\ell-1$, the conductor of $\tilde{\rho}^{(\ell)}$ is exactly $\ell$. 
\begin{proposition} \label{IJNT}
The projective representation 
$\tilde{\rho}^{(\ell)}$ has a unique linear lift 
\[
\rho^{(\ell)}: G_\Q \to \operatorname{SL}(2,\F_3)
\]
with image $\rho^{(\ell)}(G_\Q) = \operatorname{SL}(2,\F_3)$
such that $\rho^{(\ell)}$
has the same conductor, $\ell$, as $\tilde{\rho}^{(\ell)}$.
\end{proposition}

\begin{proof}
A detailed proof, using that $K$ is totally real and ramified only at one prime, together with the fundamental exact sequence of class field theory, can be found in the introduction of \cite{vangIJNT}.
\end{proof}

\subsection{Algebraicity of Hecke eigenvalues for Hecke-Maass forms}
The space of Maass cusp forms on $\Gamma_1(\ell)$ with Laplacian eigenvalue $\lambda=1/4$
lies in an irreducible discrete series representation $V_\lambda$ of $\operatorname{GL}(2,\mathbb{R})$.
The space $V_\lambda$ is necessarily finite-dimensional over $\mathbb{C}$, while the space of Maass forms in $V_\lambda$ admits a basis of simultaneous eigenvectors for all the Hecke-operators. Let $\varphi$ be such an eigenform with eigenvalues $\lambda(n)$ for integers $n\geq 1$. The automorphic $L$-function $L(\varphi,s)$ is determined by the Hecke eigenvalues at prime powers, $\lambda(p^r)$, which are algebraic integers. 

If $\varphi \in V_\lambda$ is known to correspond to a complex Galois representation
$$
\rho: G_\Q \to \operatorname{GL}(2,\mathbb{C})
$$
in the sense that $L(\varphi,s)$ equals the Artin $L$-function $L(\rho,s)$ attached to $\rho$, then at all primes $p$ at which $\rho$ is unramified, 
$
\lambda(p) = \operatorname{trace} \rho(\sigma_p)
$
for a Frobenius automorphism $\sigma_p$ at $p$.

\subsection{Reduction modulo $v=3$} \label{mod3}
At the place $v=3$, there is a more direct approach to work with the reduction modulo 3 of Maass forms corresponding to tetrahedral or octahedral Galois representations; this approach will be developed in the rest of this subsection. 

\begin{lemma}
    There exists an injective group homomorphism
    \[
 \begin{tikzcd}
 \beta:    \operatorname{SL}(2,\F_3) \arrow[hookrightarrow]{r}{\beta}
    & \operatorname{SL}(2,\Z_3)
\end{tikzcd}
\]
such that 
if
$
\pi: \operatorname{SL}(2,\Z_3) \to \operatorname{SL}(2,\Z_3)
$
is the natural homomorphism induced by reducing entries modulo 3, then 
$\pi \circ \beta$
is the identity map on $\operatorname{SL}(2,\F_3)$. 
The same statement holds if $\operatorname{SL}(2)$ is replaced everywhere with  
$\operatorname{GL}(2)$.
\end{lemma}

\begin{proof}
Consider the problem of lifting the identity map $\operatorname{Id}$ in the diagram below:
\[ 
\begin{tikzcd}
 & \operatorname{SL}(2,\Z_3) \arrow["\pi"]{d} \\
\operatorname{SL}(2,\F_3) \arrow["\operatorname{Id}"]{r} \arrow[dashed]{ur} & \operatorname{SL}(2,\F_3)     
\end{tikzcd}
\]
Let $\operatorname{Ad}^0$ be the space of traceless matrices over $\F_3$. 
Since $H^2(\operatorname{SL}(2,\F_3), \operatorname{Ad}^0) = 0$ \cite[Lemma 1, p. 566]{even1}, the map $\operatorname{Id}$ can be lifted infinitesimally 
from $\operatorname{SL}(2,\Z/3^n\Z)$ to $\operatorname{SL}(2,\Z/3^{n+1}\Z)$ for all $n \geq 1$. 
Passing to the projective limit, this gives existence of the lift $\beta$. Since $ \pi \circ \beta = \operatorname{Id} $, $\beta$ is injective. 
The fact that $H^1(\operatorname{SL}(2,\F_3), \operatorname{Ad}^0) = 0$ \cite[ibid.]{even1} implies that the lift $\beta$ is unique. 
The same argument works for $\operatorname{GL}(2)$, since 
$$
H^i(\operatorname{GL}(2,\F_3), \Ad^0) = 0
$$
for $i=1$ and $i=2$ cf. \cite[ibid.]{even1}.
\end{proof}
Let $\iota: \Z_3 \hookrightarrow \overline{\Q_3}\simeq \mathbb{C}$ be an injective ring homomorphism into $\mathbb{C}$.
By abuse of notation, we will also use the same symbol $\iota$ to denote the induced injective ring homomorphism 
$\operatorname{SL}(2,\Z_3) \hookrightarrow \operatorname{GL}(2,\mathbb{C})$. 
Define $\psi = \iota \circ \beta$ and consider the diagram 
\[
\begin{tikzcd}
    \operatorname{SL}(2,\F_3) \arrow[hookrightarrow]{dr}{\beta}
    \arrow[hookrightarrow, dashed]{rr}{\psi}&  & \operatorname{GL}(2,\mathbb{C}) \\
    & \operatorname{SL}(2,\Z_3) \arrow[hookrightarrow, swap]{ur}{\iota}
    \arrow[bend left=30, dashed]{ul}{\pi \text{ `$\bmod$ $3$'}}
\end{tikzcd}
\]

Let $\varphi$ and $\varphi'$ be Maass cusp forms of level $\Gamma_1(\ell)$ with Laplacian eigenvalue $1/4$ and of tetrahedral type. 
Let 
$$
\rho, \rho': G_\Q \to \operatorname{GL}(2,\mathbb{C})
$$ be Artin representations such that $L(\varphi,s) = \sum_{n_\geq 1} \lambda(n)n^{-s}$ equals the Artin $L$-function $L(\rho,s)$, and similarly $L(\varphi',s)= L(\rho',s)$.

Note that $\psi$ descends to an injective map $\tilde{\psi}: \operatorname{PSL}(2,\F_3) \to \operatorname{PGL}(\mathbb{C})$, 
and recall that $\varphi$ and $\varphi'$ are assumed to be of level $\Gamma_1(\ell)$. Hence 
$\rho$ and $\rho'$ both have conductor $\ell$
and projective images \[
\tilde{\rho}(G_\Q) \simeq \tilde{\rho}'(G_\Q) \simeq A_4 \simeq 
\tilde{\psi} ( \operatorname{PSL}(2,\F_3)).
\]
Up to equivalence, we may therefore assume that $\rho$ as well as $\rho'$ have projective images  
\[
\tilde{\rho}(G_\Q) = \tilde{\rho}'(G_\Q) = \tilde{\psi}(\operatorname{PSL}(2,\F_3)).
\]

By Proposition \ref{IJNT}, the projective representation $\tilde{\rho}_0:= \psi^{-1} \circ \tilde{\rho}: G_\Q \to \operatorname{PSL}(2,\F_3)$ 
has a unique lift $\rho_0$
to $\operatorname{SL}(2,\F_3)$ of the same conductor, $\ell$, as $\tilde{\rho}_0$.
Furthermore,
we know from Proposition \ref{IJNT} that the lift $\rho_0$ has image $\rho_0(G_\Q) = \operatorname{SL}(2,\F_3)$. 

Now, $\psi \circ \rho_0 : G_\Q \to \operatorname{GL}(2,\mathbb{C})$
is a lift of $\tilde{\rho}$ with the same conductor as $\tilde{\rho}$. 
By uniqueness, we must have $\psi \circ \rho_0 = \rho$. 
In particular, we have 
$\rho(G_\Q) \simeq \operatorname{SL}(2,\F_3)$ 
and equalities 
$\rho(G_\Q)  = \psi(\operatorname{SL}(2,\F_3))  =\iota \circ \beta (\operatorname{SL}(2,\F_3))$
as subgroups of $\operatorname{GL}(2,\mathbb{C})$. 
Hence $\rho(G_\Q) \subseteq \iota \operatorname{SL}(2,\Z_3)$, and 
letting $\iota^{-1}$ denote the inverse (defined on the image of $\iota$),
$$
\iota^{-1} \circ \rho: G_\Q \to \operatorname{SL}(2,\Z_3)
$$
defines a $3$-adic Galois representation into the domain of the reduction mod 3 homomorphism $\pi$. But then 
$$
\pi \circ \iota^{-1} \circ \rho = ( \iota \circ \beta )^{-1} \circ \rho = \psi^{-1}\circ \rho.
$$
By definition, we will say that $\varphi \equiv \varphi' \bmod 3$ if
$\lambda(p),\lambda'(p) \in \iota(\Z_3)$ for all $p$ and 
$$
\iota^{-1}(\lambda(p)) \equiv \iota^{-1} (\lambda'(p)) \mod 3
$$
for all but finitely many $p$. In the affirmative, 
$\pi \circ \iota^{-1} (\lambda(p)) = \pi \circ \iota^{-1} (\lambda'(p))$ and 
$
\lambda(p) = \operatorname{trace}\rho(\sigma_p) \in \mathbb{C}\cap \iota(\Z_3)
$ for all $p\not \in S$, where we define
$S = \{ 3, \infty, \ell \}$. 
Since $\iota: \Z_3\to \mathbb{C}$ is a ring homomorphism, it commutes with traces; hence
$$
\iota^{-1}(\lambda(p)) =  \operatorname{trace} \iota^{-1} \circ \rho(\sigma_p) \quad p\not \in S.
$$
Likewise,  $\pi$ is a ring homomorphism commuting with the trace map, so 
\begin{align*}
\pi \circ \iota^{-1} (\lambda(p)) &= \operatorname{trace} \pi \circ \iota^{-1} \circ \rho (\sigma_p) \\
&= \operatorname{trace}  \pi \circ \iota^{-1} \circ \rho' (\sigma_p), \quad p \not\in S.
\end{align*}
We conclude that 
$$
 \pi \circ \iota^{-1} \circ \rho =  \pi \circ \iota^{-1} \circ \rho', \quad p\not\in S.
$$
Noting that 
\begin{align*}
\psi \circ \pi \circ \iota^{-1} \circ \rho = 
\iota  \circ \iota^{-1} \circ \rho = \rho,
\end{align*}
in fact $\rho = \rho'$, so $\varphi = \varphi'$. 

\begin{lemma}[Existence of Maass forms of tetrahedral type]\label{maass}
Let $\mathbb{A}_\Q$ denote the adeles over $\Q$. 
Let $\ell \equiv 1 \mod 3$ be a prime and let $L/\Q$ be the cubic subfield of $\Q(\zeta_\ell)$. Suppose the 2-rank of the class group of $L$ is equal to $2k$ for $k\geq 1.$ 
There exist $2^k-1 $ inequivalent cuspidal automorphic representations $\pi^{(i)}$ of $\operatorname{GL}(2,\mathbb{A}_\Q)$ 
($1\leq i \leq 2^k-1$) 
corresponding to tetrahedral Maass cusp forms $\varphi^{(i)}$ 
on $\Gamma_1(\ell)$. 
In particular, if 
$$
\pi^{(i)} = \bigotimes_v \pi_v^{(i)}
$$
is the decomposition of $\pi^{(i)}$ into restricted tensor products of local representations
of $\operatorname{GL}(2,\Q_v)$, where $v$ runs over all places of $\Q$, the representation
$\pi_v^{(i)}$ is spherical for all $v\neq \ell,$
and if $i \neq j,$
the mod 3 reductions of the Maass forms remain distinct, i.e. 
    $$ 
    \varphi^{(i)} \not \equiv \varphi^{(j)} \mod 3.
    $$
In particular, if $k=1,$ there exists a unique tetrahedral Maass form on $\Gamma_1(\ell)$. 
\end{lemma}

\begin{proof}
By Proposition \ref{kA4}, there are exactly $2^k-1$ distinct $A_4$-extensions $K^{(i)}/\Q$ ramified precisely at $\ell$. These fields give rise to projective representations
$$\tilde{
\rho}^{(i)}:
\Gal(\overline{\mathbb{Q}}/\mathbb{Q}) \to \operatorname{PGL}(2,\mathbb{C})
$$
for all $i\leq 2^k-1$ such that $\tilde{\rho}^{(i)}(G_\Q)$ are isomorphic to $A_4$, where $G_\Q = \Gal(\overline{\mathbb{Q}}/\mathbb{Q})$. 
The ramification index $e_\ell$ at $\ell$ in $K^{(i)}$ is equal to 3, and since $3$ divides $\ell-1$, the conductor
of $\tilde{\rho}^{(i)}$ is equal to $\ell$ \cite[p. 248]{Serre}.
By a Theorem of Tate \cite[Theorem 5, p. 228]{Serre}, 
$\tilde{\rho}^{(i)}$ has a lifting 
$$\rho^{(i)}: 
\Gal(\overline{\mathbb{Q}}/\mathbb{Q}) \to \operatorname{GL}(2,\mathbb{C})
$$ of conductor $\ell$. 

In the language of \cite{langlands}, the representations 
$\rho^{(i)}$ are of tetrahedral type, so it follows by \cite[Theorem 3.3]{langlands} that there exist automorphic representations 
$\pi^{(i)}=\pi(\rho^{(i)})$
on $\operatorname{GL}(2)$ over $\Q$ whose Godement-Jacquet $L$-functions 
$L(\pi^{(i)},s)$
agree with the Artin $L$-functions 
$L(\rho^{(i)},s)$ 
(Tunnell \cite{tunnell} extended Langlands' work to Galois representations of octahedral type, while the tetrahedral case had been treated by Langlands in his theory of base change for $\operatorname{GL}(2)$). Since the Galois representations
$\rho^{(i)}$
are even, the automorphic representations correspond to classical Maass wave forms of the same level $\Gamma_1(\ell)$.
If the Maass forms were congruent modulo 3 as defined in Section \ref{mod3}, they would correspond to the same 
residual Galois representation over $\mathbb{F}_3$ and hence define the same $A_4$-extension, which we know is not the case.
\end{proof}

\subsection{Proof that the family of tetrahedral Maass forms is exhaustive}

\begin{lemma} \label{exh}
    Let $\pi$ be an automorphic representation of $\operatorname{GL}(2,\mathbb{A}_\Q)$ corresponding to a Maass form of tetrahedral type on $\Gamma_1(\ell)$, where $\ell$ is an odd prime.
    Then the associated $A_4$-extension $K/\Q$ is totally real, and there is a cubic subfield $L$ of $K$ such that $K/L$ is unramified. In particular, $4$ divides the class number $h_L$ of $L$. The prime $\ell$ is congruent to  $1$ modulo $3$ and $L$ is a subfield of $\Q(\zeta_\ell).$ 
\end{lemma}
\begin{proof}
The Galois group of $K/\Q$ has a unique, normal subgroup of the form $(\Z/2\Z)^2.$
Let $L$ be the fixed field of $(\Z/2\Z)^2.$
Let $\tilde{\rho}$ be the projective representation of conductor $\ell$ associated to $\pi$ via Langlands-Tunnell. 
Then $\tilde{\rho}$ is even, so $K$ is totally real. In particular, $L/\Q$ is totally real. 
By assumption, $K/\Q$ is ramified at $\ell$ and unramified at all other places. As a consequence of Hermite-Minkowski, $L/\Q$ is totally ramified at $\ell.$
The inertia group at a tame prime is cyclic, 
so the ramification index $e(\ell, K/\Q) =3.$
In particular, $K/L$ is unramified.  
\end{proof}

Theorem below summarizes the results obtained so far.

\begin{theorem}[Classification]
Let $I$ be the set of primes $\ell \equiv 1 \mod 3$ with $2$ dividing the class number of the cubic subfield of $\Q(\zeta_\ell).$
For each $\ell \in I$ there exist an integer $n_\ell^{(A_4)} \geq 1 $ and a list of tetrahedral Maass forms 
$\{ \pi^{(\ell,j)} : 1\leq j \leq n^{(A_4)}_\ell \}$ 
of level $\Gamma_1(\ell)$. In fact $n_\ell^{(A_4)} = 2^{k_\ell}-1$, where 
$k_\ell$ is half the 2-rank of the class group of the cubic subfield of $\Q(\zeta_\ell)$.

If $\pi$ is a tetrahedral Maass form of level $\Gamma_1(\ell)$ for some prime $\ell$, then
$\ell \in I$ and $\pi = \pi^{(\ell, j_\ell)}$ 
for some integer $j \leq n_\ell^{(A_4)}$. 
\end{theorem}

\begin{theorem}
For each $\ell < 7687,$ the Maass form $\pi^{(\ell)}$ is unique modulo 3. 
Among $\ell \in I$ less than $10^5,$
there exist exactly three tetrahedral Maass forms
of level $\Gamma_1(\ell)$ that are distinct modulo 3 for 
$$\ell = 7687,  16363, 20887, 37087, 55609, 62617, 70597, 99529.$$
For all other $\ell \in I$ less than $10^5$, there is only one tetrahedral Maass form $\pi^{(\ell)}$ modulo 3. 
For all primes $\ell \in I$ such that $\ell \leq 10^6,$ there exist at most three distinct tetrahedral Maass forms on $\Gamma_1(\ell)$ that remain distinct modulo 3. 

\end{theorem}
\begin{proof}
    Computations in pari-gp. In addition to the eight primes listed, there are $636$ primes $\ell \leq 10^5$ 
for which there exists one tetrahedral Maass form on $\Gamma_1(\ell)$ which is unique modulo 3.
\end{proof}

\begin{theorem}[Statistics]
For $\ell \leq 10^6,$ the proportion of primes $\ell \equiv 1 \mod 3$ which occur as the level $\Gamma_1(\ell)$
of a tetrahedral Maass form is 0.0109.
Among the $1.09\%$ of primes $\ell \leq 10^6$ that give rise to a tetrahedral Maass form on $\Gamma_1(\ell)$, $2.30\%$ of these primes give rise to three Maass forms that are distinct modulo 3. 
\end{theorem}
\begin{proof}
    Computations in pari-gp. The total number of primes less than $10^6$ that admit a Maass form of tetrahedral type on $\Gamma_1(\ell)$ is 5472, among which 126 primes give rise to two Maass forms that are distinct modulo 3. 
\end{proof}

\begin{theorem} Among primes $\ell < 2000$, the primes $\ell$ in the list
$$ I_0 := \{163, 277, 349, 547, 607, 937, 1399, 1699, 1777, 1879, 1951\}$$
are exactly those with the properties:
$\ell \equiv 1 \bmod 3$, the class number of the cubic subfield $L$ of $\Q(\zeta_\ell)$ is $4$,  
$\Gal(L_\emptyset^{(2)}/L) \simeq (\Z/2\Z)^{2}$,
there is a unique totally real $A_4$-extension $K/\Q$ such that $K/L$ is unramified and such that $L/\Q$ is ramified precisely at the place $\ell$.
The primes in $I_0$ that are not Shanks primes are exactly 
$\ell \in \{277, 547, 1399, 1699, 1777, 1879, 1951 \}$.  
\end{theorem}

\begin{theorem}\label{diagonal}
 The smallest integer $n \geq 1$ for which there exists a Maass form of tetrahedral type on $\Gamma_1(n)$ is the prime $163$.
\end{theorem}

\begin{proof}[Proof of Theorem \ref{diagonal}]
Let $\ell_1$ and $\ell_2$ be prime numbers congruent to $1$ modulo $3$.
Let $L^{(\ell_1)}$ and $L^{(\ell_2)}$ be the totally real cubic subfields of $\Q(\zeta_{\ell_1})$
and $\Q(\zeta_{\ell_1})$ respectively.
There are 4 lines in the $\F_3$-vector space $\Gal(L^{(\ell_1)}L^{(\ell_2)}/\Q)$.
Hence, inside the compositum $L^{(\ell_1)}L^{(\ell_2)}$, there are two diagonal fields $D_1$ and $D_2$ 
with conductor $\ell_1 \cdot \ell_2.$
It suffices to check that none of these diagonal fields have a 2-Hilbert class group of rank 2 or more for the list of pairs
$$
(\ell_1,\ell_2) = (7, 13), (7,19), (7,31), (7,37), (13, 19).
$$
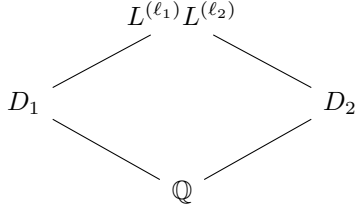
\[
\begin{tikzcd}[arrows=dash]
 & L^{(\ell_1)}L^{(\ell_2)} & \\
D_1 \arrow{ur} & & D_2 \arrow{ul} \\
& \Q \arrow[ul] \arrow[ur] & \\
\end{tikzcd}
\]
This is a straightforward computation in pari-gp.
\end{proof}

\begin{remark} \label{splitting}
For which primes $\ell$ does there exist a totally real field $L$ of class number $h_L >1$ ramified only at $\ell$ such that $\ell$ splits completely in the Hilbert class field of $L$?
The smallest example known to the author is $\ell = 163$ where $L$ can be taken as the cubic subfield of $\Q(\zeta_{163})$. 
\end{remark}
The situation described in Remark \ref{splitting} is depicted in the field diagram below for primes $\ell \equiv 1 \bmod 3$ where the class number of the cubic subfield of $\Q(\zeta_\ell)$ is even. To the left in the diagram, the splitting behavior of $\ell$ is shown: $\ell$ is totally ramified in $L/\Q$, after which the prime in $L$ above $\ell$ splits completely into four primes in $K$.
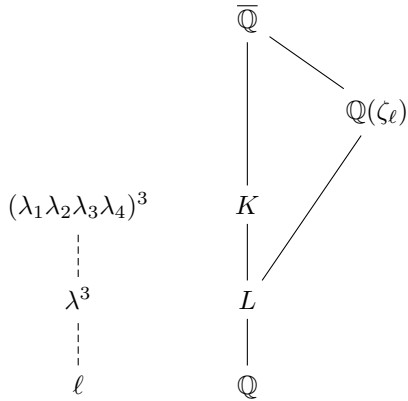
    \[
\begin{tikzcd}[arrows=dash]
 & \overline{\Q} \\
&  & \Q(\zeta_\ell) \arrow[ul] \\
(\lambda_1  \lambda_2  \lambda_3 \lambda_4)^3 & K \arrow[uu] &  \\
\lambda^3 \arrow[u, dashed]& L \arrow[u] \arrow[uur] \\
\ell \arrow[u, dashed] & \Q \arrow[u] 
\end{tikzcd}
\]

\section{Maass wave forms of octahedral type} \label{octahedralsection}
In this section, we extend the results above to the octahedral case. 

For any number field $E,$ and any finite set of places $S$ in $\Q$, recall that $E_S$ denotes the maximal extension of $E$ unramified outside all places in $E$ above $S.$

\begin{definition} \label{desc}
Let $F$ be a number field and let $L/F$ be a Galois extension.
Suppose $K$ is an extension of $L$ such that $K$ is Galois over $F$. 
We say that $K/L$ is realizable over $F$ if there is a Galois extension $K'/F$ such that 
$K = K'L$ with $\Gal(K'/F) \simeq \Gal(K/L)$ and $\Gal(K/F) \simeq \Gal(L/F)\times \Gal(K'/F).$
\end{definition}

\begin{theorem} \label{S4}
Let $\ell \equiv 1 \mod 4$ be a prime satisfying both of the following properties:
\begin{enumerate}
    \item The class number of $\Q(\sqrt{\ell})$ is divisible by 3.
    In particular, there is an unramified cubic extension $L/\Q(\sqrt{\ell})$ such that 
    $\Gal(L/\Q) \simeq S_3$.
    \item \label{thesecondassumption} When regarding $\Gal(L_{\{ \ell \} }^{(2)}/L)$ as a module over $\F_2[\Gal(L/\Q(\sqrt{\ell}))] \simeq \F_2[\Z/3\Z],$ it contains an irreducible 2-dimensional 
    Jordan-H{\"o}lder factor. 
\end{enumerate}
Then there exists a totally real extension $K/\Q$ with $\Gal(K/\Q) \simeq S_4$ unramified outside $\ell$.
\end{theorem}
\begin{proof}
    Since $\ell \equiv 1 \mod 4$, the quadratic field $\Q(\sqrt{\ell})$ is ramified precisely at $\ell$.
    Since $L/\Q(\sqrt{\ell})$ is unramified, the action of $\Gal(\Q(\sqrt{\ell})/\Q)$ on $\Gal(L/\Q(\sqrt{\ell}))$ is nontrivial. Consequently, 
    $$\Gal(L/\Q) \simeq S_3.$$
 There is a decomposition 
    $$
\Gal(L_{\{ \ell \} }^{(2)}/L) \simeq U_2^{\oplus k} \oplus \F_2^{\oplus m}
    $$
    as $\F_2[\Z/3\Z]$-modules for some $k,m\geq 0$, where $U_2$ is the irreducible 2-dimensional representation of $\Z/3\Z$ over $\F_2$, and $\F_2$ denotes the trivial representation of $\Z/3\Z$.  By assumption, $k \geq 1$; hence there is a field $K$ above $L$ such that $\Gal(K/L) \simeq U_2$ as $\F_2[\Z/3\Z]$-modules.
    Let $F = \Q(\sqrt{\ell})$.
    Note that the degree of $K/L$ is prime to the degree of
    $L/F$, so $\Gal(K/F)$ is isomorphic to a semidirect product of $\Gal(K/L)$ and $\Gal(L/F)$. However, $\Gal(L/F)$ acts nontrivially on $\Gal(K/L)$, so the semidirect product is not direct. 
    Hence $K/L$ is not realizable over $F$. 
    Reasoning as in Lemma \ref{A4}, one concludes that  
     $\Gal(K/F) \simeq A_4$. 
    Furthermore, the fact that $L/F$ is not realizable over $\Q$ implies that
    the larger extension $K/F$ is not realizable over $\Q$ as well. Hence $\Gal(K/\Q)$ is not isomorphic to the direct product $A_4 \times \Z/2\Z$. Among the 15 groups of order 24, the only possibility is that $\Gal(K/\Q)\simeq S_4.$ 
\end{proof}

\begin{proposition}[Existence of Maass forms of octahedral type]\label{maass2}
Let $\ell$ be a prime number and let  $n^{(S_4)}_\ell$ be the number of totally real $S_4$-extensions of $\Q$ ramified exactly at $\ell$. 
If $\ell \equiv 1 \mod 4$ and if the class number of $\Q(\sqrt{\ell})$ is divisible by 3, then $n_\ell^{(S_4)}\geq 1,$ and 
there exist $n_\ell^{(S_4)}$ inequivalent cuspidal automorphic representations $\pi^{(i)}$ of $\operatorname{GL}(2,\mathbb{A}_\Q),$ 
  corresponding to Maass cusp forms $\varphi^{(i)}$ of octahedral type on $\Gamma_1(\ell^k)$ where $k=1$ or $k=2$, $1\leq i \leq n_\ell^{(S_4)}.$
If 
$
\pi^{(i)} = \bigotimes_v \pi_v^{(i)}
$
is the decomposition into a restricted tensor product of local representations
of $\operatorname{GL}(2,\Q_v)$, where $v$ runs over all places of $\Q$, the representation
$\pi_v^{(i)}$ is spherical for all $v\neq \ell,$
and if $i \neq j,$
the mod 3 reductions of the Maass forms remain distinct, i.e. 
    $$ 
    \varphi^{(i)} \not \equiv \varphi^{(j)} \mod 3.
    $$
\end{proposition}

\begin{proof}
For each totally real $S_4$-extension of $\Q$ there is exactly one
equivalence class of even projective representations $\tilde{\rho}$ of octahedral type; indeed, the correspondence is as follows.  
Given $\tilde{\rho}$, the field fixed by $\ker \tilde{\rho}$ is a totally real $S_4$-extension of $\Q$. 
Conversely, given a totally real $S_4$-extension $K/\Q$,
choose an embedding $\iota: \Gal(K/\Q) \hookrightarrow \operatorname{PGL}(2,\mathbb{C})$ (unique up to conjugation in $\operatorname{PGL}(2,\mathbb{C}$)) and  
define $\tilde{\rho} = \iota \circ \pi$ to be the composition of the natural projection 
$\pi: \Gal(\overline{\Q}/\Q) \to \Gal(K/\Q)$ followed by $\iota$.
Furthermore, recall that for an octahedral projective representation $\tilde{\rho}$, a theorem of Tate shows that $\tilde{\rho}$ corresponds to a unique linear lifting $\rho$ of the same conductor \cite[Theorem 5]{Serre}.


Let $n_\ell^{(S_4)}$ be the number of totally real $S_4$-extensions of $\Q$ with corresponding even projective representations
$
\tilde{\rho}^{(i)}:
\Gal(\overline{\mathbb{Q}}/\mathbb{Q}) \to \operatorname{PGL}(2,\mathbb{C})
$
of octahedral type. Lemma \ref{S4} implies $n_\ell^{(S_4)} \geq 1$.
For each $1 \leq i\leq n_\ell^{(S_4)}$  let $\rho^{(i)}$ be the unique linear lifting of $\tilde{\rho}^{(i)}$ with the same conductor. 
By \cite{tunnell}, there are automorphic representations 
$\pi^{(i)}=\pi(\rho^{(i)})$
of $\operatorname{GL}(2)$ over $\Q$ whose Godement-Jacquet $L$-functions 
$L(\pi^{(i)},s)$
agree with the Artin $L$-functions 
$L(\rho^{(i)},s).$ 
Since  
$
S_4 \simeq \operatorname{PGL}(2,\F_3),
$
we may identify the $\tilde{\rho}^{(i)}$ with projective representations over $\F_3.$
If the Maass forms were congruent mod 3, they would correspond to the same 
projective Galois representation over $\mathbb{F}_3$ and hence define the same $S_4$-extension. 
\end{proof}

\begin{proposition} \label{exhaustive} Let $\mathbb{A}_\Q$ denote the adeles over $\Q$. 
    Let $\pi$ be an automorphic representation of $\operatorname{GL}(2,\mathbb{A}_\Q)$ corresponding to a Maass form of octahedral type on $\Gamma_1(\ell^k)$ for some prime $\ell$ and some integer $k \in \{1,2\}$. Let $K/\Q$ be the corresponding $S_4$ extension.
    Then $K/\Q$ is of the form listed in Theorem \ref{S4}.
\end{proposition}
\begin{proof}
Let $L$ be the field fixed by the normal subgroup $V_4 \simeq (\Z/2\Z)\times(\Z/2\Z)$ of $S_4.$
Then $\Gal(L/\Q) \simeq S_4/V_4 \simeq S_3.$
Let $F$ be the subfield of $L$ fixed by $A_3\simeq \Z/3\Z.$
Then $F/\Q$ is ramified only at $\ell,$ so $F=\Q(\sqrt{\ell})$ and $\ell \equiv 1 \mod 4.$

Suppose $L/\Q(\sqrt{\ell})$ was ramified at the prime above $\ell$ in $\Q(\sqrt{\ell})$.
Then $L/\Q$ would be totally ramified at $\ell$. 
But $L/\Q$ is an $S_3$-extension, while the ramification at $\ell$ is tame.
Since the inertia group at a tame prime is cyclic, we reach a contradiction. 
Hence $L/\Q(\sqrt{\ell})$ is unramified. The class number of $\Q(\sqrt{\ell})$ must therefore be divisible by 3.
\end{proof}

\begin{theorem}[Classification] \label{last}
Let $J$ be the set of primes $\ell \equiv 1 \bmod 4$ 
satisfying the two properties listed in Theorem \ref{S4}. 
Any octahedral Maass form on $\Gamma_1(\ell^k)$ for $k\in \{ 1,2\}$ occurs in the list 
\[\left\{ \pi^{(\ell)}_i: \ell \in J, i \in \{ 1,\ldots, n_\ell^{(S_4)}\}  \right\}.
\] 
\end{theorem}
\begin{proof}
By Proposition \ref{exhaustive}, the list is exhaustive. 
\end{proof}

For a number field $E$, and a set of places $S$ in $\Q$, recall that
$E_S$ denotes the maximal extension of $E$ unramified at all places in $E$ above $S$; in particular, $E_\emptyset$ is the maximal unramified extension of $E$. In addition, for a prime $p$, denote the maximal $p$-elementary abelian extension of $E$ unramified at all places in $E$ above $S$ by
$E^{(p)}_S$. If $X$ is any set, $\operatorname{card} X$ denotes the cardinality of $X$.
\begin{theorem} 
Let $F=\Q(\sqrt{\ell})$, 
$
h_\ell = \dim_{\F_3} H^1(\Gal(F_\emptyset/F) ,\F_3), 
$
and let $\mathscr{L}$ be the family of unramified cubic extensions of $F$.
Then 
$$
\operatorname{card} \mathscr{L} =\frac{ 3^{h_\ell}-1 }{2}.
$$
For each $L \in \mathscr{L}$, let
$$
k_L = \frac{\dim_{\F_2} H^1(\Gal( L_{\{ \ell\} }/L) ,\F_2) -  \dim_{\F_2} H^1(\Gal(F_{\{ \ell \} }/F),\F_2)}{2}.
$$
Then 
\begin{IEEEeqnarray*}{rCr}
n^{(S_4)}_\ell &=& 
 \sum_{L \in \mathscr{L}} (2^{k_L} -1).
\end{IEEEeqnarray*}
\end{theorem}
\begin{proof} 
Note that $h_\ell$ is the 3-rank of the class group of $F=\Q(\sqrt{\ell})$. 
Let $\mathscr{L}$ be the set of unramified cubic extensions of $F$.
Lines in the $\F_3$-vector space $H^1(\Gal(F_\emptyset/F), \F_3)$ are in bijective correspondence with $\mathscr{L}$; hence
$$
\operatorname{card}\mathscr{L} = \operatorname{card} \mathbb{P}(\F_3^{h_{\ell}}) = \frac{3^{h_\ell}- 1}{2}.
$$ 
Fix $L \in \mathscr{L}$.
In the decomposition of
$\F_2[\Gal(L/F)]$-modules 
\begin{equation} \label{decomp}
\Gal(L^{(2)}_{\{\ell \}}/L) \simeq U_2^{\oplus k} \oplus \F_2^{\oplus m},
\quad 
\dim_{\F_2} \Gal(L^{(2)}_{\{\ell \}}/L) = 2k+m, 
\end{equation}
we would like to find expressions for the integers $k,m \geq 0$. We will see that $k$ depends on $L$, while $m$ only depends on $F$.

Let $M/L$ be any quadratic extension unramified outside $\ell$. 
The action of $\Gal(L/F)$ on $\Gal(M/L) \simeq \Z/2\Z$ is trivial
if and only if $M/L$ is realizable over $F$ (using Definition \ref{desc}). 
Now, $M/L$ is realizable over $F$ precisely when 
the descended field is contained in the ray class field over $F$  whose conductor is any sufficiently large power of $\ell$.
We conclude that $m$ in the decomposition \eqref{decomp} is given by  
\[
m = \dim_{\F_2} H^1(\Gal(F_{\{ \ell \} }/F),\F_2).
\]
Consequently, the integer $k$ in \eqref{1} is given as
$$
k = \frac{\dim_{\F_2} H^1(\Gal(L_{\{\ell \}}/L),\F_2) - \dim_{\F_2} H^1(\Gal(F_{\{ \ell \} }/F),\F_2)}{2}.
$$
There are 
$$
\operatorname{card} \mathbb{P}(\F_2^{k}) = 2^{k}-1  
$$
extensions $K$ of $L$ such that $\Gal(K/L)\simeq U_2$ as
$\F_2[\Gal(L/F)]$-modules. The formula for $n^{(S_4)}_\ell$ follows. 
\end{proof}

\begin{theorem} \label{power}
Let $\ell$ be a prime. 
 If $\pi$ is a Maass form of octahedral type on $\Gamma_1(\ell^k)$ for some integer $k \geq 1$, 
 then the associated projective representation $\tilde{\rho}$ has conductor $\ell$ or $\ell^2$. 
\end{theorem}
\begin{proof}
By Theorem \ref{last}, the prime $\ell$ is in the set $J$ defined in the statement of this theorem.
Let $K/\Q$ be a totally real $S_4$ extension ramified only at $\ell$, 
and let 
$$\tilde{\rho}: \Gal(\overline{\Q}/\Q) \to \operatorname{PSL}(2,\mathbb{C})$$ 
be the corresponding even projective representation.
The ramification index $e(\ell,K/\Q)$ of $\ell$ in $K$ is either 2 or 4. 
First, suppose $e(\ell,K/\Q)=4$. 
Since $4$ divides $\ell -1$, the conductor of $\tilde{\rho}$ is exactly $\ell$ \cite[p. 248]{Serre}.
Next, suppose $e(\ell,K/\Q)=2$.
Since $\tilde{\rho}$ is tamely ramified at $\ell$, the decomposition group $\tilde{\rho}(D_\ell)$ is either cyclic or dihedral, and the conductor 
$$
N(\tilde{\rho})= \ell^{m(\ell)}
$$
of $\tilde{\rho}$ has $m(\ell)=1$ if and only if the decomposition group $\tilde{\rho}(D_\ell)$ is cyclic.
\end{proof}

\bibliographystyle{amsplain} 
\bibliography{references.bib}

@article{even1,
  author =       "R. Ramakrishna",
  title =        "{Deforming an even representation}",
  journal =      "Inventiones mathematicae",
  volume =       "322",
  number =       "3",
  pages =        "563--580",
  year =         "1998",
  label = "even1"
}

@misc{repZ3,
  title =  "Linear representation theory of cyclic group:{Z}3",
author = "Groupprops",
note = "Available at \href{https://groupprops.subwiki.org/wiki/Linear_representation_theory_of_cyclic_group:Z3}{groupprops.subwiki.org} (accessed 2023-10-21)"
}

@article{shanks,
  author =       "D. Shanks",
  title =        "{The simplest cubic fields}",
  journal =      "Mathematics of Computation",
  volume =       "28",
  number =       "128",
  pages =        "1137-1152",
  year =         "1974"
}

@book{langlands,
  author =       "R. P. Langlands",
  title =        "Base change for $\operatorname{GL}(2)$",
  year =         1980,
  publisher = "Princeton University Press" 
}

@incollection{gelbart,
  author      = "S. Gelbart",
  title       = "Modularity and the {L}anglands Reciprocity Conjecture",
  editor      = "Cornell, G. and Silverman, J. and Stevens, G.",
  booktitle   = "Modular forms and Fermat's last theorem",
  publisher   = "Springer",
  year        = "1997",
  pages       = "155-208",
  chapter     = "VI",
}

@article{tunnell,
  author =       "J. Tunnell",
  title =        "{Artin's conjecture for representations of octahedral type}",
  journal =      "Bulletin of the American Mathematical Society",
  volume =       "5",
  number =       "2",
  pages =        "173-175",
  year =         "1981"
}

@misc{conrad,
  author =       "K. Conrad",
  title =        "\href{https://kconrad.math.uconn.edu/blurbs/grouptheory/group12.pdf}{Groups of order 12}",
  note = "Accessed: 2023-09-03"
}

@incollection{Calegari,
    AUTHOR = {Calegari, F.},
     TITLE = {Reciprocity in the {L}anglands program since {F}ermat's last
              theorem},
 BOOKTITLE = {I{CM}---{I}nternational {C}ongress of {M}athematicians. {V}ol.
              2. {P}lenary lectures},
     PAGES = {610--651},
 PUBLISHER = {EMS Press, Berlin},
      YEAR = {2023}
}

@book {Bump,
    AUTHOR = {Bump, D.},
     TITLE = {Automorphic forms and representations},
    SERIES = {Cambridge Studies in Advanced Mathematics},
    VOLUME = {55},
 PUBLISHER = {Cambridge University Press, Cambridge},
      YEAR = {1997},
     PAGES = {xiv+574}
}

@article{FIMR,
  author =       "Friedlander, J. B. and Iwaniec, H. and Mazur, B. and Rubin, K.",
  title =        "{The spin of prime ideals}",
  journal =      "Inventiones mathematicae",
  volume =       "193",
  number =       "",
  pages =        "697–749",
  year =         "2013"
}

@incollection {Serre,
    AUTHOR = {Serre, J.-P.},
     TITLE = {Modular forms of weight one and {G}alois representations},
 BOOKTITLE = {Algebraic number fields: {$L$}-functions and {G}alois
              properties ({P}roc. {S}ympos., {U}niv. {D}urham, {D}urham,
              1975)},
     PAGES = {193--268},
 PUBLISHER = {Academic Press, London-New York},
      YEAR = {1977}
}

@article{vangIJNT,
    author = {Uttenthal, P. V.},
    title = {A classification of even representations onto 3-adic {SL}(2)},
    journal = {International Journal of Number Theory},
volume =  {21},
number = {10},
year = {2025},
pages = {2441-2460},
    note = { \href{https://doi.org/10.1142/S1793042125501179}{doi.org/10.1142/S1793042125501179} }
}


\end{document}